\documentclass[reqno,12 pt]{amsart}
\usepackage{amsmath,amssymb,latexsym,geometry}
\usepackage{amsmath}
\usepackage{amssymb}
\usepackage{latexsym}
\usepackage{amsthm}
\usepackage{hyperref}
\usepackage{color}
\hypersetup{linkcolor=blue, colorlinks=true ,citecolor = red}
\usepackage{graphicx}

\begin{document}

	\title[Existence of Multiple Solutions ]{Multiple solutions for a weighted $\MakeLowercase{p}$-Laplacian problem}

	\author[R. Kumar \& A. Sarkar]{Rohit Kumar,  Abhishek Sarkar}
		\address{The Department of Mathematics,  IIT Jodhpur, Jodhpur, Rajasthan 342030, India}
	\email{kumar174@iitj.ac.in, abhisheks@iitj.ac.in}

	

	\begin{abstract}
		We prove the existence of at least three solutions for a weighted $p$-Laplacian operator involving Dirichlet boundary condition in a weighted Sobolev space.  The main tool we use here is a three solution theorem in reflexive Banach spaces due to G. Bonanno and B. Ricceri.
	\end{abstract}

\maketitle
\numberwithin{equation}{section}
\newtheorem{theorem}{Theorem}[section]
\newtheorem{lemma}[theorem]{Lemma}
\newtheorem{proposition}[theorem]{Proposition}
\newtheorem{corollary}[theorem]{Corollary}
\newtheorem{definition}[theorem]{Definition}
\newtheorem{example}[theorem]{Example}
\newtheorem{remark}[theorem]{Remark}
\allowdisplaybreaks
\section{Introduction}
In this article we are interested in proving existence of three solutions for a Dirichlet boundary value problem involving weighted $p$-Laplacian operator. We consider the following problem: 
\begin{align}\label{main}
\begin{cases}
-\mathrm{div}(a(x)|\nabla u|^{p-2}\nabla u)  +|u|^{p-2}u = \lambda f(x,u) +\mu g(x,u) \ \ \text{ in } \Omega,\\
\qquad u=0 \text{ on } \partial \Omega,
\end{cases}
\end{align}
where $\Omega \subset \mathbb{R}^N$ is a bounded domain, $p>1$ and $N \geq 1$. The restriction between $p$ and $N$ will be specified as we proceed. 

We assume that the weight $a$ satisfies the following conditions
\begin{align} \label{weight}
\begin{cases}
a \text{ is positive a.e. in } \Omega,\\
a^{-1/(p-1)} \in L^1_{\mathrm{loc}}(\Omega),\\
a \in L^1_{\mathrm{loc}}(\Omega); \ a^{-s} \in L^1(\Omega) \text{ with some } s>0.
\end{cases}
\end{align}
We look for solutions in the weighted Sobolev space $W^{1,p}_0(a;\Omega)$ associated with the weight $a(x)$, which is defined in Section~\ref{Sec2}.

The weighted operator was first introduced by Murthy-Stampacchia \cite{Murthy} for the second order linear pdes. Later it was generalized to higher order linear pdes and also quasilinear elliptic pdes. For interested reader we refer to the book by Dr\'abek et al. \cite{Drabek} and the research article \cite{Le}, where boundary value problems for weighted $p$-Laplacian operators have been studied independently.    
Our aim is to show the existence of at least three solutions of problem \eqref{main}, by using a \emph{three critical points theorem} introduced by Riccieri and also by Bonanno in their series of articles. First we state the theorem proved by Riccieri \cite{Riccieri}.
\begin{theorem}
	Let $X$ be a separable and reflexive real Banach space; $I \subset \mathbb{R}$ an interval; $\phi:X \to \mathbb{R}$ a sequentially weakly lower semicontinuous $C^1$ functional whose derivative admits a continuous inverse on $X^{\ast}$; $J:X \to \mathbb{R}$ a $C^1$ functional with compact derivative. Assume that
	\[ \lim \limits_{\|u\| \to \infty} (\phi(u)+\lambda J(u)) = +\infty,\]
	for all $\lambda \in I$, and that there exists $\rho \in \mathbb{R}$ such that 
	\[ \sup \limits_{\lambda \in I} \inf_{u \in X} (\phi(u) + \lambda (J(u)+\rho)) < \inf \limits_{u \in X} \sup_{\lambda \in I} (\phi(u) +\lambda(J(u)+\rho)).\]
	Then, there exist a non-empty open set $\Gamma \subset I$ and a positive real number $r$ such that, for each $\lambda \in \Gamma$, the equation
	\[\phi'(u)+\lambda J'(u)=0,\]
	has at least three solutions in $X$ whose norms are less than $r$.
\end{theorem}
\par We note that the first result appeared in the literature due to Ricceri \cite{Ricci1st}, having made assumptions that the space is reflexive and separable. Later, Bonanno \cite{Bon31} gave an equivalent conditions to Ricceri's theorem. But Ricceri \cite{Ricci2nd} then generalized his result only for reflexive Banach spaces (with some compensation). Here we state the equivalent theorem combining \cite{Bon31, Ricci2nd}.
\begin{theorem}\label{threecritical}
	Let $X$ be a reflexive Banach space;  $\phi:X \to \mathbb{R}$ a continuously G\^ateaux differentiable and sequentially weakly lower semicontinuous $C^1$ functional, bounded on each bounded subset of $X$, whose G\^ateaux derivative admits a continuous inverse on $X^*$; $\Phi:X \to \mathbb{R}$ a $C^1$  functional with compact G\^ateaux derivative. Assume that 
	\begin{itemize}
		\item[(i)] $\lim \limits_{\|u\| \to \infty} (\phi(u) + \lambda \Phi(u)) = +\infty$;
		\item[(ii)] there exist $r \in \mathbb{R}$ and $u_0, u_1 \in X$ such that $\phi(u_0) <r<\phi(u_1)$; 
		\item[(iii)] $\displaystyle \inf \limits_{u \in \phi^{-1}((-\infty,r])} \Phi(u) > \frac{(\phi(u_1)-r)\Phi(u_0)+(r-\phi(u_0))\Phi(u_1)}{\phi(u_1)-\phi(u_0)}.$ 
	\end{itemize}
	Then, there exists a non-empty open set $\Gamma \subset [0,\infty)$ and a positive real number $\rho$ with the following property: for each $\lambda \in \Gamma$ and every $C^1$ functional  $J:X \to \mathbb{R}$ with compact G\^ateaux derivative, there exists $\delta >0$ such that for each  $\mu \in [0,\delta]$, the equation
	$$\phi'(u)+\lambda \Phi'(u)+\mu J'(u)=0$$
	has at least three solutions in $X$, whose norms are less than $\rho.$
\end{theorem}
As an application of aforementioned theorems we refer to \cite{AH, Livrea} for Dirichlet boundary value problems and for Neumann boundary value problems we refer to \cite{Anello, Candito} and the references therein. We follow the similar path to \cite{Livrea}. 

\par The rest of this paper is organized as follow. In Section~\ref{Sec2}, we briefly discuss the weighted Sobolev spaces and state the main theorem. Section~\ref{Results} deals with the proof of main theorem and also some necessary lemmas.
\section{Preliminaries and Result}\label{Sec2}

We briefly discuss the weighted Sobolev spaces in a way the approach had been done in \cite{Drabek}. Given $a$ satisfying \eqref{weight}, the weighted Sobolev space $W^{1,p}(a;\Omega)$ is defined to be the set of all real valued measurable functions $u$ for which 
\begin{equation}\label{norm1}
\|u\| := \bigg(\int_{\Omega} |u|^p \mathrm{d} x + \int_{\Omega} a(x)|\nabla u|^p \mathrm{d} x\bigg)^{1/p} <\infty.
\end{equation}
Since $a^{-1/(p-1)} \in L^1_{\mathrm{loc}}(\Omega)$ (see \eqref{weight}), it follows that $W^{1,p}(a;\Omega)$ equipped with the norm $\|\cdot\|$ is uniformly convex Banach space; thus by Milman--Pettis theorem it is a reflexive Banach space. The assumption $a \in L^1_{\mathrm{loc}}(\Omega)$ (see \eqref{weight}) ensures 
\[ C_0^{\infty}(\Omega) \subset W^{1,p}(a;\Omega),\] which allows us to consider the closure of $C_0^{\infty}(\Omega)$ with respect to the norm $\| \cdot \|$, and denote it by $W^{1,p}_0(a;\Omega)$. 
Moreover, the continuous embedding holds
\begin{equation}\label{contemb}
W^{1,p}_0(a;\Omega) \hookrightarrow W^{1,p_s}_0(\Omega), \ \text{ where } p_s:= \frac{ps}{s+1}. 
\end{equation}
Note that, $p >p_s$. When $p_s >N$, from the classical Sobolev embedding theorem we have the following compact embedding:
\begin{equation}\label{embed} W^{1,p}_0(a;\Omega)  \hookrightarrow W^{1,p_s}_0(\Omega) \hookrightarrow \hookrightarrow C^{0,\alpha}(\bar{\Omega}),\end{equation}
for all $0 < \alpha <1-(N/p_s).$ Hereafter, it is always assumed that $s>0$ (in \eqref{weight}) is chosen such that \[ \fbox{$p>p_s>N$ \text{ i.e., } $s > N/(p-N)$.} \]
\begin{remark} \label{eqnorm} It is worth mentioning that by recalling a version of Friedrichs type inequality associated with some weight (see \cite[eq. no (1.28), p.27]{Drabek}) the norm 
	\[\|u\|_a:= \bigg( \int_{\Omega} a(x) |\nabla u|^p \mathrm{d}x \bigg)^{1/p} ,\]
	on the space $W^{1,p}_0(a;\Omega)$ is equivalent to the norm $\|\cdot\|$ defined in \eqref{norm1}.\end{remark}
\begin{example}A typical example of weight \eqref{weight} can be considered as 
	\[ a(x) := \frac{1}{\mathrm{dist}(x,\partial \Omega)^{l}},\] for $l \geq 0$; where '$\mathrm{dist}$' denotes the distance function from a point $x$ in $\Omega$ to the boundary $\partial \Omega$.  
\end{example}
\par From the above embedding \eqref{embed}, we have
\begin{equation}\label{k-bound}
k:= \sup_{u \in W^{1,p}_0(a;\Omega) \setminus\{0\}} \mathrm{\frac{max_{\bar{\Omega}} |u(x)|}{\|u\|} } < \infty. \end{equation}
\begin{remark}
	Note that, we can talk about an upper bound for above $k$. Using the above embedding \eqref{contemb} and \cite{talenti} it follows that 
	\[ k \leq \frac{N^{-1/p_s}}{\sqrt{\pi}} \bigg[\Gamma\bigg(1+\frac{N}{2}\bigg)\bigg]^{1/N} \bigg( \frac{p_s-1}{p_s-N}\bigg)^{1-1/p_s} |\Omega|^{1/N-1/p_s}. \]
\end{remark}
\begin{definition}
	A \emph{weak solution} of problem \eqref{main} is such an $u \in W^{1,p}_0(a;\Omega)$ which satisfies 
	\begin{align}\label{weaksol}
	\int_{\Omega} a(x)|\nabla u|^{p-2} \nabla u \cdot \nabla v \mathrm{d} x + \int_{\Omega}|u|^{p-2} u v \mathrm{d} x &= \lambda \int_{\Omega} f(x,u)v \mathrm{d} x  +\mu \int_{\Omega} g(x,u)v \mathrm{d} x,
	\end{align}
	for every $v \in W^{1,p}_0(a;\Omega)$. 
\end{definition}
\par Fix $x_0 \in \Omega$ and choose $r_1, r_2$ with $0<r_1<r_2$ such that $B(x_0,r_1) \subset B(x_0, r_2) \subset \subset \Omega$, where $B(x,r)$ denotes the ball in $\mathbb{R}^N$  centered at $x$ and of radius $r$. Let 
\begin{equation} \xi = \xi(p,r_1,r_2):= \frac{2 k  r_1}{(r_2^2 -r_1^2)} \|a\|_{L^1(A_{r_1}^{r_2})}^{1/p},  \label{lbu} \end{equation} and 
\begin{equation} \eta=\eta(p,N,r_1,r_2) := \bigg( \frac{2^p k^p  r_2^p }{(r_2^2-r_1^2)^p} \|a\|_{L^1(A_{r_1}^{r_2})} + \frac{ k^p d^p}{N}w_N r_2^N +  k^p w_N r_1^N\bigg)^{1/p}, \label{ubu}\end{equation} where $A_{r_1}^{r_2}:= B(x_0,r_2) \setminus B(x_0,r_1).$
We note that $\xi$ and $\eta$ both are finite since $a \in L^1_{\mathrm{loc}}(\Omega)$. 
We also define $$F(x,t):= \int_{0}^{t} f(x,s) \mathrm{d} s \text{ and } G(x,t) := \int_{0}^t g(x,s) \mathrm{d} s.$$
\begin{theorem}[Main Result]\label{result}
	Assume that there exist three positive constants $c,d$ and $\gamma$ with $d^p \xi^p >c^p$ and functions $h$ and $w_{\tau}\in L^1(\Omega)$ such that
	\begin{itemize}
		\item[(H1)] $F(x,t) \geq 0$ for each $(x,t) \in \{\bar{\Omega} \setminus B(x_0,r_1)\} \times [0,d]$
		\item[(H2)] $ d^p \eta^p |\Omega| \displaystyle\sup_{(x,t) \in \Omega \times [-c, c]} F(x,t) < c^p \int_{\Omega} F(x,d) \mathrm{d} x$
		\item[(H3)] $F(x,t) < h(x) (1 + |t|^{\gamma}),$ for a.e. $x \in \Omega$ and $t \in \mathbb{R}$ large
		\item[(H4)] $F(x,0) =0$ for a.e $x \in \Omega$
		\item[(H5)] $g:\Omega \times \mathbb{R} \to \mathbb{R}$ be a Carath\'eodory function  such that for all $\tau>0$, there 
		\[\displaystyle \sup_{|t| \leq \tau} |g(\cdot,t)| \leq w_{\tau}(x).\]  
	\end{itemize}
	Then, there exists an open interval $\Lambda \subset [0,\infty)$ and a positive real number $\rho$ with the following property: for each $\lambda \in \Lambda$ 
	there exists $\delta >0$ such that for each $\mu \in [0,\delta]$, problem \eqref{main} has at least three weak solutions in $W^{1,p}_0(a;\Omega)$, whose norms are less than $\rho$. 
\end{theorem}

\section{Proof of Main Result}\label{Results}
In this section we prove the main result and necessary lemmas.
We define the following functionals $\phi, \Phi$ and $ \Upsilon$ on  $W^{1,p}_0(a;\Omega)$ by 
\begin{align*}
\phi(u)&:= \frac{1}{p} \int_{\Omega} a(x) |\nabla u|^p \mathrm{d} x + \frac{1}{p} \int_{\Omega} |u|^p \mathrm{d} x = \frac{1}{p} \|u\|^p,\\ 
\Phi(u)&:= - \int_{\Omega} F(x,u) \mathrm{d} x\ \text{ and } 
\Upsilon(u):= -\int_{\Omega} G(x,u) \mathrm{d} x.
\end{align*}
It is worth mentioning that since $p_s>N$ and together with the assumptions on $f$ and $g$, the functionals $\Phi$ and $\Upsilon$ are well defined. 
Then for any $u, v \in W^{1,p}_0(a;\Omega)$, we have 
\begin{align*}
(\phi'(u),v) &= \int_{\Omega} a(x)|\nabla u|^{p-2} \nabla u \cdot \nabla v \mathrm{d} x + \int_{\Omega}|u|^{p-2} u v \mathrm{d}x,\\
(\Phi'(u),v) &= -\int_{\Omega} f(x,u) v \mathrm{d} x \ \text{ and }
(\Upsilon'(u),v) = -\int_{\Omega} g(x,u)v \mathrm{d} x.
\end{align*}

From \eqref{weaksol} it is clear that $u \in W^{1,p}_0(a;\Omega)$ be a weak solution of problem \eqref{main} if for every $v \in W^{1,p}_0(a;\Omega)$ following identity holds
\begin{equation*}
(\phi'(u),v) + \lambda (\Phi'(u),v) + \mu (\Upsilon'(u),v) =0.
\end{equation*}

Thus, we can look for solutions (weak) of problem \eqref{main} applying Bonanno's theorem for three solutions. 

\begin{lemma}[Continuous Inverse]\label{inverse}
	It follows that $(\phi')^{-1} : X^{\ast} \to X$ exists and it is continuous. 
\end{lemma}
\begin{proof}
	For any  $x,y \in \mathbb{R}^N$ then by applying the inequality from  \cite{Veron}, 
	there holds
	\[\langle |x|^{p-2} x -|y|^{p-2} y, x-y \rangle \geq \frac{1}{2^p} |x-y|^p, \  p \geq 2, \] for all $x,y \in \mathbb{R}^N$, where $\langle \cdot, \cdot \rangle$ denotes the usual inner product in $\mathbb{R}^N$.
	
	Thus, noting that $a(x)>0$ a.e., we have
	\[(\phi'(u) - \phi'(v), u-v) \geq c_p \|u-v\|^p, \ \ \forall u,v \in W^{1,p}_0(a;\Omega),\] for $p \geq 2$. Hence $\phi'$ is uniformly monotone operator in $W^{1,p}_0(a;\Omega),\ p \geq 2$. Also, for the case $1<p<2$, we can proceed as in \cite[Lemma 4]{Le} and get the desired uniform monotonicity.
		 In addition, a simple computation suggests that $\phi'$ is coercive. Indeed, 
	\[\frac{(\phi'(u),u)}{\|u\|} \geq \frac{\|u\|^p}{\|u\|} = \|u\|^{p-1}. \]
	Also note that the map $t \mapsto (\phi'(u+tv),w)$ is continuous on $[0,1]$ for all $u,v,w \in W^{1,p}_0(a;\Omega)$, hence $\phi'$ is hemicontinuous. Therefore, the conclusion follows immediately by applying Theorem 26.A of \cite{Zeidler}
\end{proof}
Next, we prove another lemma which is essential to prove Theorem \ref{result}. 
\begin{lemma}\label{ustar}
	Assume that there exist two positive constants $c,d$ with $ d^p \xi^p>c^p$ such that 
	\begin{itemize}
		\item[(F1)] $F(x,t) \geq 0$ for each $(x,t) \in \{\Omega \setminus B(x_0,r_1)\} \times [0,d].$
		\item[(F2)] $d^p \eta^p |\Omega| \displaystyle\sup_{\bar{\Omega} \times [-c,c]} F(x,t) < c^p \int_{\Omega} F(x,d) \mathrm{d} x.$
	\end{itemize}
	Then there exist $r>0$ and $u^{\ast} \in W^{1,p}_0(a;\Omega)$, such that 
	\begin{equation}\label{bonanno1}
	\phi(u^{\ast})=\frac{1}{p} \|u^{\ast}\|^p  >r,
	\end{equation}
	and
	\begin{equation} \label{bona1}
	|\Omega| \max_{\bar{\Omega} \times [-c,c]} F(x,t) \leq \bigg(\frac{c}{k \|u^{\ast}\|}\bigg)^{p} \int_{\Omega} F(x,u^{\ast})\mathrm{d} x.
	\end{equation}
\end{lemma}
\begin{proof}
	Define 
	\begin{align*} 
	u^{\ast}(x) = \begin{cases}
	d, \hspace{3.76cm}x \in B(x_0,r_1),\\
	\frac{d}{(r_2^2 -r_1^2)} (r_2^2-|x-x_0|^2),  \ x \in B(x_0,r_2) \setminus B(x_0,r_1),  \\
	0, \hspace{3.765cm} x \in \Omega \setminus B(x_0,r_2).
	\end{cases}
	\end{align*}
	It is easy to check that $u^{\ast} \in W^{1,p}_0(a;\Omega)$. Note that, 
	\begin{align}\|u^{\ast}\|^p = \frac{2^p d^p}{(r_2^2 -r_1^2)^p} \int_{A_{r_1}^{r_2}} a(x) |x-x_0|^p \mathrm{d} x &+ \frac{w_Nd^p}{(r_2^2 -r_1^2)^p} \int_{r_1}^{r_2} (r_2^2-r^2)^p r^{N-1} \mathrm{d} r \notag\\ &+ d^pw_N r_1^N.   \label{uStar}
	\end{align}
	From \eqref{lbu}, \eqref{ubu} and \eqref{uStar}, we deduce 
	\begin{eqnarray}\label{bonaineq}
	\frac{\xi^p d^p}{k^p} < \|u^{\ast}\|^p <\frac{\eta^p d^p}{k^p}.
	\end{eqnarray} Now by using the fact that $d^p \xi^p >c^p$ and \eqref{bonaineq}, we get 
	\begin{align} \label{bon1}
	\frac{1}{p} \|u^{\ast}\|^p > \frac{1}{p} \frac{\xi^p d^p}{k^p} > \frac{1}{p} \frac{c^p}{k^p}.
	\end{align} By choosing $r:= \displaystyle \frac{1}{p} \bigg(\frac{c}{k}\bigg) ^p$, \eqref{bonanno1} follows from \eqref{bon1} immediately. 
	Since $0 \leq u^{\ast} \leq d$ for each $x \in \Omega$, the condition (F1) suggests
	\begin{equation}\label{F1}
	\int_{\Omega \setminus B(x_0,r_2)} F(x,u^{\ast}(x)) \mathrm{d} x + \int_{B(x_0,r_2) \setminus B(x_0,r_1)} F(x,u^{\ast}(x)) \mathrm{d} x \geq 0.
	\end{equation}
	Now by using condition (F2), \eqref{F1} and the definition of $u^{\ast}$, we get
	\begin{align*}
	|\Omega| \max_{(x,t) \in \bar{\Omega} \times [-c,c]} F(x,t) &< \bigg(\frac{c}{\eta d}\bigg)^p \int_{B(x_0,r_1)} F(x,d) \mathrm{d} x \\
	&= \bigg(\frac{c}{\eta d}\bigg)^p \int_{B(x_0,r_1)} F(x,u^{\ast}) \mathrm{d} x\\
	&< \bigg(\frac{c}{k\|u^{\ast}\|}\bigg)^p \int_{B(x_0,r_1)} F(x,u^{\ast}) \mathrm{d} x \\
	&\leq  \bigg(\frac{c}{k\|u^{\ast}\|}\bigg)^p \int_{\Omega} F(x,u^{\ast}) \mathrm{d} x, 
	\end{align*}
	i.e., \eqref{bona1} follows. The proof is complete.
\end{proof} 
Now we give a proof of our main theorem of this article. 
\begin{proof}[Proof of Theorem~\ref{result}]
	First we note down following observations which are of immediate consequences: 
	\begin{itemize}
		\item[(i)] $\Phi$ belongs to $C^1$ and also $\Phi'$ is compact.
		\item[(ii)] $\phi$ is weakly lower semicontinuous (since it's a norm) and bounded on each bounded subset of $W^{1,p}_0(a;\Omega)$. 
		\item[(iii)] $(\phi')^{-1}$ exists and is continuous too, thanks to Lemma~\ref{inverse}.
		\item[(iv)] From the assumptions on $g$, it follows that $\Upsilon$ is continuously G\^ateaux differentiable on $W^{1,p}_0(a;\Omega)$, with compact derivative. 
	\end{itemize}
	Thanks to (H3), for each $\lambda \geq 0$, we have 
	\begin{equation*} \lim_{\|u\| \to \infty} (\phi(u) +\lambda \Phi (u)) = +\infty. \end{equation*}
	Put $r= \displaystyle \frac{1}{p} \bigg(\frac{c}{k}\bigg)^p.$ Note that $\displaystyle\max_{\bar{\Omega}} |u(x)| \leq k \|u\|, $ for every $u \in W^{1,p}_0(a;\Omega)$. Hence for each $u$ such that 
	\[\phi(u) =\frac{1}{p}\|u\|^p \leq r \] and one has $$\max_{\bar{\Omega}}|u(x)| \leq k \|u\| =c.$$ 
	Thanks to Lemma~\ref{ustar}, there exists $u^{\ast} \in W^{1,p}_0(a;\Omega)$ such that 
	\[\phi(u^{\ast})= \frac{1}{p}\|u^{\ast}\|^p >r >0=\phi(0).\] 
	Therefore using \eqref{k-bound} and \eqref{bona1}, we get
	\begin{align}
	-\inf_{u \in \phi^{-1}((-\infty,r])} \Phi(u) &= \sup_{u \in \phi^{-1}((-\infty,r])} (-\Phi(u)) \leq \sup_{\{u: \|u\|^p \leq pr\}} \int_{\Omega} F(x,u) \mathrm{d} x \notag\\
	&<\int_{\Omega}\sup_{ |t| \leq c }F(x,t) \mathrm{d} x  < |\Omega| \max_{\bar{\Omega} \times [-c,c]} F(x,t) \notag \\
	&<\bigg(\frac{c}{k\|u^{\ast}\|}\bigg)^p \int_{\Omega} F(x,u^{\ast}) \mathrm{d} x  \notag\\
	&= \frac{1}{p} \bigg(\frac{c}{k}\bigg)^p \frac{p}{\|u^{\ast}\|^p} \int_{\Omega} F(x,u^{\ast}) \mathrm{d} x \notag\\
	&= r \frac{p}{\|u^{\ast}\|^p}  \int_{\Omega} F(x,u^{\ast}) \mathrm{d} x  = \frac{r(-\Phi(u^{\ast}))}{\phi(u^{\ast})}.  \label{finalineq}
	\end{align}
	From \eqref{finalineq}, one obtains 
	\[\sup_{u \in \phi^{-1}((-\infty,r])} (-\Phi(u)) <\frac{r(-\Phi(u^{\ast}))}{\phi(u^{\ast})}, \] or 
	\begin{equation}\label{bonanno2}
	\inf_{u \in \phi^{-1}((-\infty,r])} \Phi(u) > \frac{r\Phi(u^{\ast})}{\phi(u^{\ast})}.
	\end{equation}
 	Now we choose $u_0=0$ and  $u_1=u^{\ast}$, so that $\Phi(u_0)=0=\phi(u_0)$ and from \eqref{bonanno2}, we get 
	\[\inf_{u \in \phi^{-1}((-\infty,r])}\Phi(u) > \frac{(\phi(u_1)-r)\Phi(u_0) +(r-\phi(u_0))\Phi(u_1)}{\phi(u_1)-\phi(u_0)}.\]
	Hence all the conditions of Theorem \ref{threecritical} are satisfied and the existence of three nontrivial distinct solutions follow immediately. 
\end{proof}
\begin{remark}
	We note that the method is still applicable for other boundary conditions as well. For example, we can consider the same problem \eqref{main} with Neumann boundary condition, and look for the solution in the space $W^{1,p}(a;\Omega)$. 
\end{remark}
\begin{remark}
	Remark \ref{eqnorm} hints that we can also consider the following boundary value problem and discuss about the existence of at least three solutions in $W^{1,p}_0(a;\Omega)$ for the following Dirichlet boundary value problem: 
	\begin{align}
	\begin{cases}
	-\mathrm{div}(a(x)|\nabla u|^{p-2}\nabla u)   = \lambda f(x,u) +\mu g(x,u) \ \ \text{ in } \Omega,\\
	u=0 \text{ on } \partial \Omega,
	\end{cases}
	\end{align}
	where $\Omega \subset \mathbb{R}^N$ is a bounded domain.
\end{remark}
\section*{Acknowledgment} The first author acknowledges the support of the CSIR fellowship for his Ph.D. work. A part of the research was conducted while the second author was in University of West Bohemia, Pilsen, Czech Republic. The second author was supported the Grant Agency of the Czech Republic,  project no. 18-03253S and also by the DST-INSPIRE Grant DST/INSPIRE/04/2018/002208

\end{document}